\documentclass[12pt]{article}
\usepackage{amssymb,amsmath,amsthm,amsfonts,enumerate, bbm, mathdots}

\usepackage{latexsym}
\usepackage{amscd}
\usepackage{stmaryrd}
\usepackage{color}
\usepackage[all]{xypic}
\usepackage{epsfig}
\usepackage{graphics}
\usepackage{ifthen}
\usepackage{varioref}
\usepackage{rotating}
\usepackage{extarrows}
\usepackage{cite}
\usepackage{mathrsfs}

\numberwithin{equation}{section}

\theoremstyle{plain}
\newtheorem{thm}{Theorem}[section]
\newtheorem{cor}[thm]{Corollary}
\newtheorem{lem}[thm]{Lemma}

\newmuskip\pFqmuskip
\newcommand*\pFq[6][8]{%
  \begingroup 
  \pFqmuskip=#1mu\relax
  \mathchardef\normalcomma=\mathcode`,
  \mathcode`\,=\string"8000
  \begingroup\lccode`\~=`\,
  \lowercase{\endgroup\let~}\pFqcomma
  {}_{#2}F_{#3}{\left(\genfrac..{0pt}{}{#4}{#5};#6\right)}%
  \endgroup
}
\newcommand{\pFqcomma}{{\normalcomma}\mskip\pFqmuskip}

\makeatletter
\newenvironment{proofof}[1]{\par
  \pushQED{\qed}%
  \normalfont \topsep6\p@\@plus6\p@\relax
  \trivlist
  \item[\hskip\labelsep
        \bfseries
    Proof of #1\@addpunct{.}]\ignorespaces
}{%
  \popQED\endtrivlist\@endpefalse
}
\makeatother

\definecolor{dg}{rgb}{0.0625,0.64,0.0625}
\usepackage[
            pdfstartview=FitH,
            CJKbookmarks=true,
            bookmarksnumbered=true,
            bookmarksopen=true,
            colorlinks,
            pdfborder=001,
            linkcolor=blue,
            anchorcolor=green,
            citecolor=red
            ]{hyperref}

\definecolor{deepmaroon}{RGB}{100,0,20}     
\definecolor{deepnavy}{RGB}{0,30,80}        
\definecolor{deepforest}{RGB}{0,60,30}      
\definecolor{deeppurple}{RGB}{128,0,255}     
\definecolor{deepteal}{RGB}{0,70,70}        
\definecolor{deepolive}{RGB}{60,65,30}      
\definecolor{darkredheavy}{RGB}{120,0,0}     
\definecolor{lightgreen}{RGB}{0,255,255}    
\definecolor{darkgreenheavy}{RGB}{0,80,0}    
\definecolor{darkpurpleheavy}{RGB}{90,0,120} 
\definecolor{darktealheavy}{RGB}{0,90,90}    
\definecolor{darkorangeheavy}{RGB}{160,60,0} 
\definecolor{matcha}{RGB}{130,190,140}
\definecolor{emeraldDeep}{RGB}{0,115,90}
\definecolor{grassBright}{RGB}{20,180,60}

\newfont{\scyr}{wncyr10 scaled 550}

\def\proof{\noindent {\bf Proof.\;}}

\def\wt{\operatorname{wt}}
\def\dep{\operatorname{dep}}
\def\height{\operatorname{ht}}

\allowdisplaybreaks

\begin{document}

\title{Weighted sum formulas for finite multiple mixed values}
\date{~}
\author{{Zhonghua Li${}^{a,}$\thanks{Email: zhonghua\_li@tongji.edu.cn}\quad and\quad{Zhenlu Wang${}^{b,}$\thanks{Email: zhenluwang@tzc.edu.cn}}}\\[1mm]\small a. School of Mathematical Sciences,\\ \small Key Laboratory of Intelligent Computing and Applications (Tongji University), \\ \small Ministry of Education, \small Tongji University, Shanghai 200092, China\\
\small b. School of Artificial Intelligence, \\ \small Taizhou University, Taizhou 318000, Zhejiang, China}


\maketitle

\begin{abstract}
In this paper, we employ the iterated integral expression of multiple polylogarithms to establish a weighted sum formula for finite multiple mixed values. As applications, we derive a one-parameter weighted sum formula for finite multiple $T$-values of arbitrary even depth, as well as several relations for depth-two, height-zero and height-one finite multiple mixed values.
\end{abstract}

{\small

}


\section{Introduction}
\subsection{Multiple zeta values and level-two variants}
A finite sequence of positive integers is called an index. An index $\boldsymbol{k}=(k_1,\ldots,k_r)$ is called admissible if $k_r\geq2$. The weight, depth and height of $\boldsymbol{k}$ are defined by $\wt(\boldsymbol{k}):=k_1+\cdots+k_r$, $\dep(\boldsymbol{k}):=r$, and $\height(\boldsymbol{k}):=|\{i|k_i\geq2\}|$ respectively.

For any admissible index $\boldsymbol{k}=(k_1,\ldots,k_r)$, Hoffman\cite{Hoffman1992} and Zagier\cite{Zagier1994} independently defined the multiple zeta value (MZV) $\zeta(\boldsymbol{k})$ by 
\begin{align*}
  \zeta(\boldsymbol{k}):=\sum\limits_{0<m_1<\cdots<m_r}\frac{1}{m_1^{k_1}\cdots m_r^{k_r}}.
\end{align*}

These values have featured prominently in both number theory and physics, attracting considerable attention and interest over the past three decades. One of the main points of interest in this area is to find as many relations among MZVs as possible and further explore the structure of the $\mathbb{Q}$-vector space spanned by MZVs. There are many linear relations among MZVs. A well-known example is the sum formula \cite{Hoffman1992}
\begin{align*}
  \sum\limits_{k_1+\cdots+k_r=k\atop k_r\geq2}\zeta(k_1,\ldots,k_r)=\zeta(k),\quad (k\geq 2).
\end{align*}
Ohno and Zudilin \cite{OhnoZudilin} proved the following weighted sum formula for double zeta values:
\begin{align*}
  \sum\limits_{k_1+k_2=k\atop k_2\geq2}2^{k_2}\zeta(k_1,k_2)=(k+1)\zeta(k),\quad (k\geq 3).
\end{align*}
 Guo and Xie \cite{GuoXie2009} generalized the above formula to arbitrary depth. Many other types of weighted sum formulas among MZVs exist; see Zhao's monograph \cite{Zhao2016} for further details.

In recent years, a number of variants of MZVs have been introduced and studied. Among them, a family known as level-two MZVs arises from restricting the summation in the definition of MZVs to a prescribed parity pattern. Prominent examples include Hoffman's multiple $t$-values (MtVs)\cite{Hoffman2019}, Kaneko-Tsumura's multiple $T$-values (MTVs)\cite{Kaneko-Tsumura2020}, and Xu-Zhao's multiple $S$-values (MSVs)\cite{XuZhao2022}. For any admissible index $\boldsymbol{k}=(k_1,\ldots,k_r)$, we give their respective definitions as follows:
\begin{align*}
&t(\boldsymbol{k})=t(k_1,\dots,k_r):=2^r\sum\limits_{0<m_1<\cdots<m_r\atop \forall m_i:odd}\frac{1}{m_1^{k_1}\cdots m_r^{k_r}},\\
&T(\boldsymbol{k})=T(k_1,\dots,k_r):=2^r\sum\limits_{0<m_1<\cdots<m_r\atop m_i\equiv i \pmod2}\frac{1}{m_1^{k_1}\cdots m_r^{k_r}},\\
&S(\boldsymbol{k})=S(k_1,\dots,k_r):=2^r\sum\limits_{0<m_1<\cdots<m_r\atop m_i\equiv i-1 \pmod2}\frac{1}{m_1^{k_1}\cdots m_r^{k_r}}.
\end{align*}
Note that the definition of multiple $t$-values adopted in this paper coincides with that in \cite{Hoffman2019}, with an extra multiplicative factor of $2^r$.

To reformulate the definitions of these level-two MZVs, Xu and Zhao\cite{XuZhao2022} introduced the multiple mixed values (MMVs, or multiple M-values). For any admissible index $\boldsymbol{k}=(k_1,\ldots,k_r)$ and any finite sequence $\boldsymbol{\varepsilon}=(\varepsilon_1,\ldots,\varepsilon_r)\in\{\pm 1\}^r$, the multiple mixed value $M(\boldsymbol{k};\boldsymbol{\varepsilon})$ is defined by 
\begin{align*}
M(\boldsymbol{k};\boldsymbol{\varepsilon}):=\sum\limits_{0<m_1<\dots< m_r}\frac{(1+\varepsilon_1(-1)^{m_1})\cdots(1+\varepsilon_r(-1)^{m_r})}{m_1^{k_1}\cdots m_r^{k_r}}=2^r\sum_{\substack{0 < m_1 < \dots < m_r \\ 2 \mid m_j \text{ if } \varepsilon_j = 1 \\ 2 \nmid m_j \text{ if } \varepsilon_j = -1}}\frac{1}{m_1^{k_1} \cdots m_r^{k_r}}.
\end{align*}

By applying Chen's theory of iterated integrals, one can define a shuffle algebra that reflects the multiplication rule of two MZVs in terms of these integrals. For any differential 1-forms $f_1(u)du, \dots, f_n(u)du$, the iterated integral is defined by 
\begin{align*}
\int_0^z f_1(u)du \cdots f_n(u)du = \int\limits_{0<u_1<\cdots<u_n<z} f_1(u_1)du_1 \cdots f_n(u_n)du_n.
\end{align*}
To investigate the shuffle structure of MMVs, we define the following 1-forms: 
\begin{align*}
\omega_0(u) = \frac{du}{u}, \quad 
\omega_1(u) = \frac{2u du}{1-u^2},\quad
\omega_{-1}(u) = \frac{2du}{1-u^2}. 
\end{align*}
As established in \cite{XuZhao2022}, MMVs admit the following integral representations:
\begin{align*}
M(k_1, \dots, k_r; \varepsilon_1, \varepsilon_2, \dots, \varepsilon_r) = \int_0^1 \omega_{\varepsilon_1}\omega_0^{k_1-1} \omega_{\varepsilon_1\varepsilon_2 }\omega_0^{k_2-1}\cdots\omega_{\varepsilon_{r-1}\varepsilon_r }\omega_0^{k_r-1}
\end{align*}
and 
\begin{align*}
\int_0^1 \omega_{\varepsilon_1}\omega_0^{k_1-1} \omega_{\varepsilon_2 }\omega_0^{k_2-1}\cdots\omega_{\varepsilon_r}\omega_0^{k_r-1}= M(k_1, \dots, k_r; \varepsilon_1, \varepsilon_1 \varepsilon_2, \dots, \varepsilon_1 \cdots \varepsilon_r). 
\end{align*}

\subsection{Finite multiple zeta values and  level-two variants}

Let $\mathcal{P}$ be the set of primes. Set
\begin{align*}
\mathcal{A} := \prod_{p \in \mathcal{P}} (\mathbb{Z}/p\mathbb{Z}) \Big/ \bigoplus_{p \in \mathcal{P}} (\mathbb{Z}/p\mathbb{Z}).
\end{align*}
For any index $\boldsymbol{k}=(k_1,\ldots,k_r)$,  the finite multiple zeta value (FMZV) $\zeta_{\mathcal{A}}(\boldsymbol{k})$ was introduced by Kaneko and Zagier (see \cite{Kaneko2019,KanekoZagier-prep}) via the definition
\begin{align*}
\zeta_{\mathcal{A}}(\boldsymbol{k}):=\left( \sum_{0<m_1<\dots<m_r<p} \frac{1}{m_1^{k_1}\cdots m_r^{k_r}} \pmod{p} \right)_{p \in \mathcal{P}} \in \mathcal{A}.
\end{align*}
Kaneko and Zagier proposed a central conjecture on finite multiple zeta values, which predicts a deep connection between finite and classical multiple zeta values. Similarly, there are many linear relations among finite multiple zeta values. It is known that FMZVs satisfy the reversal formula
\begin{align*}
\zeta_\mathcal{A}(k_1,\ldots,k_r)=(-1)^{k_1+\cdots+k_r}\zeta_\mathcal{A}(k_r,\ldots,k_1).
\end{align*} 
Saito and Wakabayashi \cite{SaitoWakabayashi} proved a general identity, from which the following analogue of the classical sum formula can be derived as a special case:
\begin{align*}
  \sum\limits_{k_1+\cdots+k_r=k\atop k_r\geq2}\zeta_\mathcal{A}(k_1,\ldots,k_r)=\left(1+(-1)^r\binom{k-1}{r-1}\right)Z(k),
\end{align*}
where for an integer $k\geq2$, $Z(k)\in\mathcal{A}$ is defined by 
\begin{align*}
  Z(k)_{(p)}:=\frac{B_{p-k}}{k}\pmod p\quad\text{for}\;p>k,\quad (B_{p-k}: \text{Bernoulli number}).
\end{align*}
Using the iterated integral expression for multiple polylogarithms, Kamano \cite{Kamano2018} derived some weighted sum formulas for FMZVs.

In this paper, we focus on the level-two finite multiple zeta values. For any index $\boldsymbol{k}=(k_1,\ldots,k_r)$, any finite sequence $\boldsymbol{\varepsilon}=(\varepsilon_1,\ldots,\varepsilon_r)\in\{\pm 1\}^r$, and any positive integer $n$, Zhao \cite{Zhao2024b} defined the $n$-th partial multiple mixed sum $M_n(\boldsymbol{k};\boldsymbol{\varepsilon})$ by 
\begin{align*}
M_n(\boldsymbol{k};\boldsymbol{\varepsilon}):&=\sum\limits_{0<m_1<\dots<m_r<n}\frac{(1+\varepsilon_1(-1)^{m_1})\cdots(1+\varepsilon_r(-1)^{m_r})}{m_1^{k_1}\cdots m_r^{k_r}}\\&=2^r\sum_{\substack{0< m_1<\dots<m_r<n \\ 2 \mid m_j \text{ if } \varepsilon_j = 1 \\ 2 \nmid m_j \text{ if } \varepsilon_j = -1}}\frac{1}{m_1^{k_1} \cdots m_r^{k_r}}
\end{align*}
and the finite multiple mixed value (FMMV) $M_\mathcal{A}(\boldsymbol{k};\boldsymbol{\varepsilon})$ by 
\begin{align*}
M_\mathcal{A}(\boldsymbol{k};\boldsymbol{\varepsilon}):=\Bigl(M_p(\boldsymbol{k};\boldsymbol{\varepsilon}) \pmod{p}\Bigr)_{p\in\mathcal{P}}\in\mathcal{A}.
\end{align*}

Finite multiple mixed values serve as a common generalization of the following four level-two variants of finite multiple zeta values. For any index $\boldsymbol{k}=(k_1,\ldots,k_r)$, the finite multiple $t$-value (FMtV), the finite multiple $E$-value (FMEV), the finite multiple $T$-value (FMTV), and the finite multiple $S$-value (FMSV) are defined by 
\begin{align*}
&t_\mathcal{A}(\boldsymbol{k})=t_\mathcal{A}(k_1,\dots,k_r):=M_\mathcal{A}(k_1,\dots,k_r;\{-1\}^r),\\
&E_\mathcal{A}(\boldsymbol{k})=E_\mathcal{A}(k_1,\dots,k_r):=M_\mathcal{A}(k_1,\dots,k_r;\{1\}^r),\\
&T_\mathcal{A}(\boldsymbol{k})=T_\mathcal{A}(k_1,\dots,k_r):=M_\mathcal{A}(k_1,\dots,k_r;-1,1,\dots,(-1)^{r-1},(-1)^r),\\
&S_\mathcal{A}(\boldsymbol{k})=S_\mathcal{A}(k_1,\dots,k_r):=M_\mathcal{A}(k_1,\dots,k_r;1,-1,\dots,(-1)^{r-2},(-1)^{r-1}),
\end{align*}
respectively, where $\{a\}^n$ denotes the string $a,\ldots,a$ with $n$ repetitions. We remark that our definitions of these level-two variants match those in \cite{Zhao2024b} up to a multiplicative factor of $2^r$. Furthermore, FMEVs satisfy $E_\mathcal{A}(\boldsymbol{k})=2^{\dep(\boldsymbol{k})-\wt(\boldsymbol{k})}\zeta_\mathcal{A}^{(2)}(\boldsymbol{k})$, where $\zeta_\mathcal{A}^{(2)}(\boldsymbol{k})$ was introduced by Kaneko et al. \cite{KanekoMurakamiYoshihara2023}. They established a parity result and several sum formulas, and also conjectured a weighted sum formula for indices whose components belong to $\{1,2\}$. This conjecture was recently proved by the present authors together with Zhang \cite{LiWangZhang2026}. Zhao \cite{Zhao2024a,Zhao2024b} studied these level-two variants of FMZVs, establishing their stuffle, reversal, and linear shuffle relations, which can be used to compute generating sets in weights one and two and to analyze the dimensions of the various subspaces generated by these values. 

For any tuple $\boldsymbol{d}$ and any real number $c$, let $\overleftarrow{\boldsymbol{d}}$ denote the reversal of $\boldsymbol{d}$, and let $c\boldsymbol{d}$ denote the tuple formed by multiplying each component of $\boldsymbol{d}$ by $c$. In particular, $-\boldsymbol{d}$ denotes the tuple obtained by negating all components of $\boldsymbol{d}$. Similarly, the FMMVs admit the following reversal relation \cite{Zhao2024b}:
  \begin{align}
    M_\mathcal{A}(\overleftarrow{\boldsymbol{k}};-\overleftarrow{\boldsymbol{\varepsilon}})=(-1)^{\wt(\boldsymbol{k})}M_\mathcal{A}(\boldsymbol{k};\boldsymbol{\varepsilon}).\label{Eq:reversal relation}
  \end{align}

In this paper, we provide a weighted sum formula for finite multiple mixed values.  The following is our main theorem, where the notation $\sum_{|\boldsymbol{k}|_{r}=N}$ indicates that the sum is taken over all indices $\boldsymbol{k}=(k_1,\ldots,k_r)$ satisfying $k_1+\cdots k_r=N$. 

\begin{thm}\label{Thm:main result}
Let $a,b$ be non-negative integers, and let $\delta_1,\delta_2,\tau\in\{\pm 1\}$. For all integers $m,n$ satisfying $0\leq m\leq a$ and $0\leq n\leq b$, we have
  \begin{align*}
  &\sum\limits_{0\leq i\leq a\atop 0\leq j\leq b}\binom{i}{m}\binom{j}{n}\\
  &\qquad\qquad\times\sum\limits_{\substack{|\boldsymbol{k}|_{a-i+1}=a+b-i-j+1\\|\boldsymbol{l}|_{i+1}=i+j+1}}M_\mathcal{A}(\boldsymbol{k},\boldsymbol{l};\delta_1,\delta_1\tau,\ldots,\delta_1\tau^{a-i},\delta_1\delta_2\tau^{a-i},\ldots,\delta_1\delta_2\tau^{a-1},\delta_1\delta_2\tau^{a})\\
  &+\sum\limits_{0\leq i\leq a\atop 0\leq j\leq b}\binom{i}{a-m}\binom{j}{b-n}\\
  &\qquad\qquad\times\sum\limits_{\substack{|\boldsymbol{k}|_{a-i+1}=a+b-i-j+1\\|\boldsymbol{l}|_{i+1}=i+j+1}}M_\mathcal{A}(\boldsymbol{k},\boldsymbol{l};\delta_2,\delta_2\tau,\ldots,\delta_2\tau^{a-i},\delta_1\delta_2\tau^{a-i},\ldots,\delta_1\delta_2\tau^{a-1},\delta_1\delta_2\tau^{a})\\
  &=\sum\limits_{\substack{|\boldsymbol{k}|_{a-m+1}=a+b-m-n+1\\|\boldsymbol{l}|_{m+1}=m+n+1}}(-1)^{m+n+1}\\
  &\qquad\qquad\qquad\qquad\qquad\times M_\mathcal{A}(\boldsymbol{k},\boldsymbol{l};\delta_1,\delta_1\tau,\ldots,\delta_1\tau^{a-m},-\delta_2\tau^{m},-\delta_2\tau^{m-1},\ldots,-\delta_2\tau,-\delta_2).
  \end{align*}
\end{thm}

This theorem is proved by using the iterated integral expression for multiple polylogarithms. The same method has been used to give various weighted sum formulas for classical MZVs and FMZVs (see \cite{GuoXie2009, OngEieLiaw, Kamano2018}).

Among the applications of Theorem \ref{Thm:main result}, an important one is a one-parameter weighted sum formula for finite multiple $T$-values of arbitrary even depth. For an index $\boldsymbol{k}=(k_1,\ldots,k_r)$, define
\begin{align*}
\mathcal{W}_d(\boldsymbol{k}):=k_{r-d+1}+\cdots+k_r,\qquad 1\leq d\leq r,
\end{align*}
which is the sum of the last $d$ components of $\boldsymbol{k}$. We then obtain the following result.

\begin{thm}\label{thm:wtsum-z-T}
Let $z$ be a parameter. For $w\geq r\geq2$ with $r$ even, we have
\begin{align}
\sum_{|\boldsymbol{k}|_r=w}\sum_{d=1}^{r-1}\left[(1+z)^{\mathcal{W}_d(\boldsymbol{k})-1}
\bigl(1+z^{w-\mathcal{W}_d(\boldsymbol{k})-1}\bigr)+(-z)^{\mathcal{W}_d(\boldsymbol{k})-1}\right]
T_\mathcal{A}(\boldsymbol{k})=0.\label{eq:wtsum-T}
\end{align}
\end{thm}

Theorem \ref{Thm:main result} also yields several further applications. By specializing Theorem \ref{Thm:main result} to depth two, we obtain a general weighted sum formula for depth-two FMMVs, from which corresponding weighted sum formulas for finite multiple $E$-, $S$-, $t$-, and $T$-values are obtained by appropriate choices of $\delta_1$ and $\delta_2$. Moreover, the height-zero specialization gives explicit relations for FMMVs with certain two-block sign patterns, while the height-one specialization yields further weighted sum relations for finite multiple $t$- and $T$-values.

The paper is organized as follows. We prove Theorem \ref{Thm:main result} in Section \ref{Sec:Proof}. In Section \ref{Sec:Applications}, we discuss several applications of our main result. More precisely, Subsection \ref{subsec:wtsum T-values} is devoted to the proof of Theorem \ref{thm:wtsum-z-T} and some of its corollaries; Subsection \ref{subsec:depth-two} treats the depth-two case; and Subsection \ref{subsec:height-zero and one} is concerned with the height-zero and height-one cases.


\section{Proof of Theorem \ref{Thm:main result}}\label{Sec:Proof}
For any index $\boldsymbol{k}=(k_1,\ldots,k_r)$ and any finite sequence $\boldsymbol{\varepsilon}=(\varepsilon_1,\ldots,\varepsilon_r)\in\{\pm 1\}^r$, we define the multiple polylogarithm by
\begin{align*}
\mathrm{Mi}(\boldsymbol{k};\boldsymbol{\varepsilon};z):=\sum\limits_{0<m_1<\dots< m_r}\frac{(1+\varepsilon_1(-1)^{m_1})\cdots(1+\varepsilon_r(-1)^{m_r})z^{m_r}}{m_1^{k_1}\cdots m_r^{k_r}}.
\end{align*}

For any prime $p$, as defined in \cite{Kamano2018}, the $\mathbb{Q}$-linear operator $\mathcal{C}_p:\mathbb{Q}[[z]]\to\mathbb{Q}$ is given by
\begin{align*}
  \mathcal{C}_p\sum\limits_{n=0}^\infty a_nz^n:=\sum\limits_{n=0}^{p-1}a_n.
\end{align*}

\begin{lem}\label{Lem:stuffle-relation}
For indices $\boldsymbol{k}=(k_1,\ldots,k_r),\boldsymbol{l}=(l_1,\ldots,l_s)$, and any finite sequences $\boldsymbol{\varepsilon}=(\varepsilon_1,\ldots,\varepsilon_r)\in\{\pm 1\}^r$, $\boldsymbol{\delta}=(\delta_1,\ldots,\delta_s)\in\{\pm 1\}^s$, we have
\begin{align}
\mathcal{C}_p\left(\mathrm{Mi}(\boldsymbol{k};\boldsymbol{\varepsilon};z)\mathrm{Mi}(\boldsymbol{l};\boldsymbol{\delta};z)\right)\equiv(-1)^{\wt(\boldsymbol{l})}M_p(\boldsymbol{k},\overleftarrow{\boldsymbol{l}};\boldsymbol{\varepsilon},-\overleftarrow{\boldsymbol{\delta}})\pmod p.\label{Eq:stuffle-relation}
\end{align}
\end{lem}
\proof The left-hand side of \eqref{Eq:stuffle-relation} is equal to
\begin{align*}
  \sum\limits_{\substack{0<m_1<\cdots<m_r\\ 0<n_1<\cdots<n_s\\m_r+n_s<p}}\frac{(1+\varepsilon_1(-1)^{m_1})\cdots(1+\varepsilon_r(-1)^{m_r})(1+\delta_1(-1)^{n_1})\cdots(1+\delta_s(-1)^{n_s})}{m_1^{k_1}\cdots m_r^{k_r}n_1^{l_1}\cdots n_s^{l_s}}.
\end{align*}
By making the change of variables $n_j\to p-n_j$ for $1\leq j\leq s$, we have
\begin{align*}
  &\mathcal{C}_p\left(\mathrm{Mi}(\boldsymbol{k};\boldsymbol{\varepsilon};z)\mathrm{Mi}(\boldsymbol{l};\boldsymbol{\delta};z)\right)\\
  &=\sum_{\substack{0<m_1<\cdots<m_r\\0<p-n_1<\cdots<p-n_s\\m_r < n_s}}\frac{\prod_{i=1}^r\bigl(1+\varepsilon_i(-1)^{m_i}\bigr)\prod_{j=1}^s \bigl(1+\delta_j(-1)^{p-n_j}\bigr)}{\prod_{i=1}^r m_i^{k_i}\prod_{j=1}^s(p-n_j)^{l_j} }\\
  &\equiv \sum_{\substack{0<m_1<\cdots<m_r<n_s<\cdots<n_1<p}} \frac{\prod_{i=1}^r (1+\varepsilon_i(-1)^{m_i}) \prod_{j=1}^s (1-\delta_j(-1)^{n_j})}{(-1)^{l_1+\cdots+l_s} \prod_{i=1}^r m_i^{k_i} \prod_{j=1}^s n_j^{l_j}} \pmod p\\
  &=(-1)^{\wt(\boldsymbol{l})}M_p(\boldsymbol{k},\overleftarrow{\boldsymbol{l}};\boldsymbol{\varepsilon},-\overleftarrow{\boldsymbol{\delta}}),
\end{align*}
and we obtain \eqref{Eq:stuffle-relation}.
\qed

We define the following integrals:
\begin{align*}
&\Omega_0(u_1,u_2):=\int_{u_1}^{u_2}\omega_0(u)=\int_{u_1}^{u_2}\frac{du}{u},\\
&\Omega_1(u_1,u_2):=\int_{u_1}^{u_2}\omega_1(u)=\int_{u_1}^{u_2}\frac{2udu}{1-u^2},\\
&\Omega_{-1}(u_1,u_2):=\int_{u_1}^{u_2}\omega_{-1}(u)=\int_{u_1}^{u_2}\frac{2du}{1-u^2}.
\end{align*}
We may deduce the following lemma via a similar approach to that employed in \cite[Proposition 2.1]{Eie-Wei2008} and \cite[Lemma 2.2]{Kamano2018}.

\begin{lem}\label{Lem:shuffle-relation}
For a positive integer $n$, let $i_m, j_m$ be non-negative integers ($1 \leq m \leq n$) and $\tau_m, \delta_m \in \{\pm 1\}$. Then we have
\begin{align*}
&\frac{1}{i_1! \cdots i_n! j_1! \cdots j_n!}
\int\limits_{0 < u_1 < \cdots < u_n < z}
\Omega_{\tau_1}(u_1,u_2)^{i_1} \Omega_{\tau_2}(u_2,u_3)^{i_2} \cdots \Omega_{\tau_n}(u_n,z)^{i_n} \\
& \quad \Omega_0(u_1,u_2)^{j_1} \Omega_0(u_2,u_3)^{j_2} \cdots \Omega_0(u_n,z)^{j_n}\omega_{\delta_1}(u_1)\cdots\omega_{\delta_n}(u_n) \\
=&\sum_{\substack{|\boldsymbol{k}_m|_{i_m+1}= i_m + j_m + 1\\ \boldsymbol{\varepsilon}_m=\delta_1\cdots\delta_m\tau_1^{i_1}\cdots\tau_{m-1}^{i_{m-1}}(1,\tau_m,\ldots,\tau_m^{i_m})\\ m=1,\ldots,n}}
\mathrm{Mi}(\boldsymbol{k}_1, \dots, \boldsymbol{k}_n; \boldsymbol{\varepsilon}_1,\dots,\boldsymbol{\varepsilon}_n;z).
\end{align*}
\end{lem}

Combining the integral and series representations of FMMVs, we prove the following theorem using the lemmas established above.
\begin{thm} 
Let $\alpha_1,\alpha_2,\beta_1,\beta_2$ be parameters. Then for any non-negative integers $a,b$ and any $\delta_1,\delta_2,\tau\in\{\pm 1\}$, we have
\begin{align}
&\sum\limits_{\substack{i_1+i_2=a\\j_1+j_2=b\\i_1,i_2,j_1,j_2\geq0}}\alpha_1^{i_1}(\alpha_1+\alpha_2)^{i_2}\beta_1^{j_1}(\beta_1+\beta_2)^{j_2}\notag\\
  &\qquad\times\sum_{\substack{|\boldsymbol{k}|_{i_1+1}=i_1+j_1+1\\|\boldsymbol{l}|_{i_2+1}=i_2+j_2+1}}M_\mathcal{A}(\boldsymbol{k},\boldsymbol{l}; \delta_1,\delta_1\tau,\ldots,\delta_1\tau^{i_1},\delta_1\delta_2\tau^{i_1},\delta_1\delta_2\tau^{i_1+1},\ldots,\delta_1\delta_2\tau^{i_1+i_2})\notag\\
&+\sum\limits_{\substack{i_1+i_2=a\\j_1+j_2=b\\i_1,i_2,j_1,j_2\geq0}}\alpha_2^{i_1}(\alpha_1+\alpha_2)^{i_2}\beta_2^{j_1}(\beta_1+\beta_2)^{j_2}\notag\\
&\qquad\times\sum_{\substack{|\boldsymbol{k}|_{i_1+1}=i_1+j_1+1\\|\boldsymbol{l}|_{i_2+1}=i_2+j_2+1}}M_\mathcal{A}(\boldsymbol{k},\boldsymbol{l}; \delta_2,\delta_2\tau,\ldots,\delta_2\tau^{i_1},\delta_1\delta_2\tau^{i_1},\delta_1\delta_2\tau^{i_1+1},\ldots,\delta_1\delta_2\tau^{i_1+i_2})\notag\\
=&\sum\limits_{\substack{i_1+i_2=a\\j_1+j_2=b\\i_1,i_2,j_1,j_2\geq0}}\alpha_1^{i_1}\alpha_2^{i_2}\beta_1^{j_1}\beta_2^{j_2}\notag\\
&\qquad\times\sum_{\substack{|\boldsymbol{k}|_{i_1+1}=i_1+j_1+1\\|\boldsymbol{l}|_{i_2+1}=i_2+j_2+1}}(-1)^{i_2+j_2+1}M_\mathcal{A}(\boldsymbol{k},\boldsymbol{l}; \delta_1,\delta_1\tau,\ldots,\delta_1\tau^{i_1},-\delta_2\tau^{i_2},-\delta_2\tau^{i_2-1},\ldots,-\delta_2).\label{Eq:wtsum}
\end{align}
\end{thm}
\proof
We consider the following function:
\begin{align*}
F(z)=\frac{1}{a!\,b!}\int\limits_{\substack{0 < u_1< z \\ 0 < u_2 < z}}\left( \alpha_1 \Omega_\tau(u_1,z) + \alpha_2 \Omega_\tau(u_2,z) \right)^a\left( \beta_1 \Omega_0(u_1,z) + \beta_2 \Omega_0(u_2,z) \right)^b\omega_{\delta_1}(u_1)\omega_{\delta_2}(u_2).
\end{align*}
By using the binomial expansion and Lemma \ref{Lem:shuffle-relation}, we get
\begin{align*}
&F(z)\\&=\sum\limits_{\substack{i_1+i_2=a\\j_1+j_2=b\\i_1,i_2,j_1,j_2\geq0}}\frac{\alpha_1^{i_1}\alpha_2^{i_2}\beta_1^{j_1}\beta_2^{j_2}}{i_1!\,i_2!\,j_1!\,j_2!}\int\limits_{\substack{0 < u_1< z \\ 0 < u_2 < z}}\Omega_\tau(u_1,z)^{i_1}\Omega_\tau(u_2,z)^{i_2}\Omega_0(u_1,z)^{j_1}\Omega_0(u_2,z)^{j_2}\omega_{\delta_1}(u_1)\omega_{\delta_2}(u_2)\\
&=\sum\limits_{\substack{i_1+i_2=a\\j_1+j_2=b\\i_1,i_2,j_1,j_2\geq0}}\alpha_1^{i_1}\alpha_2^{i_2}\beta_1^{j_1}\beta_2^{j_2}\sum_{\substack{|\boldsymbol{k}|_{i_1+1}=i_1+j_1+1\\|\boldsymbol{l}|_{i_2+1}= i_2+j_2+1}}
\mathrm{Mi}(\boldsymbol{k}; \delta_1,\delta_1\tau,\ldots,\delta_1\tau^{i_1};z)\mathrm{Mi}(\boldsymbol{l}; \delta_2,\delta_2\tau,\ldots,\delta_2\tau^{i_2};z).
\end{align*}
Applying $\mathcal{C}_p$ and using Lemma \ref{Lem:stuffle-relation}, we obtain
\begin{align*}
&(\mathcal{C}_pF(z))_p\\&=\sum\limits_{\substack{i_1+i_2=a\\j_1+j_2=b\\i_1,i_2,j_1,j_2\geq0}}\alpha_1^{i_1}\alpha_2^{i_2}\beta_1^{j_1}\beta_2^{j_2}\\
&\qquad\times\sum_{\substack{|\boldsymbol{k}|_{i_1+1}=i_1+j_1+1\\|\boldsymbol{l}|_{i_2+1}= i_2+j_2+1}}(-1)^{i_2+j_2+1}M_{\mathcal{A}}(\boldsymbol{k},\boldsymbol{l};\delta_1,\delta_1\tau,\ldots,\delta_1\tau^{i_1},-\delta_2\tau^{i_2},-\delta_2\tau^{i_2-1},\ldots,-\delta_2).
\end{align*}
 
On the other hand, by partitioning the region of integration, we can rewrite $F(z)$ as
\begin{align}
  F(z)= \frac{1}{a!\,b!}\left(\displaystyle\int_{0<u_1<u_2<z}+\displaystyle\int_{0<u_2<u_1<z}\right).\label{Eq:DividIntegral}
\end{align}
We denote the first integral in \eqref{Eq:DividIntegral} by $F_1(z)$ and the second by $F_2(z)$.
It is clear that 
\begin{align*}
\alpha_1 \Omega_\tau(u_1,z) + \alpha_2 \Omega_\tau(u_2,z)=\alpha_1 \Omega_\tau(u_1,u_2) + (\alpha_1+\alpha_2) \Omega_\tau(u_2,z),\\
\beta_1 \Omega_0(u_1,z) + \beta_2 \Omega_0(u_2,z)=\beta_1 \Omega_0(u_1,u_2) + (\beta_1+\beta_2) \Omega_0(u_2,z).
\end{align*}
Then we have 
\begin{align*}
  F_1(z)&=\sum\limits_{\substack{i_1+i_2=a\\j_1+j_2=b\\i_1,i_2,j_1,j_2\geq0}}\frac{\alpha_1^{i_1}(\alpha_1+\alpha_2)^{i_2}\beta_1^{j_1}(\beta_1+\beta_2)^{j_2}}{i_1!\,i_2!\,j_1!\,j_2!}\\
  &\qquad\times\int\limits_{\substack{0<u_1<u_2<z}}\Omega_\tau(u_1,u_2)^{i_1}\Omega_\tau(u_2,z)^{i_2}\Omega_0(u_1,u_2)^{j_1}\Omega_0(u_2,z)^{j_2}\omega_{\delta_1}(u_1)\omega_{\delta_2}(u_2)\\
  &=\sum\limits_{\substack{i_1+i_2=a\\j_1+j_2=b\\i_1,i_2,j_1,j_2\geq0}}\alpha_1^{i_1}(\alpha_1+\alpha_2)^{i_2}\beta_1^{j_1}(\beta_1+\beta_2)^{j_2}\\
  &\qquad\times\sum_{\substack{|\boldsymbol{k}|_{i_1+1}=i_1+j_1+1\\|\boldsymbol{l}|_{i_2+1}= i_2+j_2+1}}\mathrm{Mi}(\boldsymbol{k},\boldsymbol{l}; \delta_1,\delta_1\tau,\ldots,\delta_1\tau^{i_1},\delta_1\delta_2\tau^{i_1},\delta_1\delta_2\tau^{i_1+1},\ldots,\delta_1\delta_2\tau^{i_1+i_2};z).
\end{align*}
Applying $\mathcal{C}_p$, we obtain
\begin{align*}
  (\mathcal{C}_pF_1(z))_p&=\sum\limits_{\substack{i_1+i_2=a\\j_1+j_2=b\\i_1,i_2,j_1,j_2\geq0}}\alpha_1^{i_1}(\alpha_1+\alpha_2)^{i_2}\beta_1^{j_1}(\beta_1+\beta_2)^{j_2}\\
  &\quad\times\sum_{\substack{|\boldsymbol{k}|_{i_1+1}=i_1+j_1+1\\|\boldsymbol{l}|_{i_2+1}= i_2+j_2+1}}M_{\mathcal{A}}(\boldsymbol{k},\boldsymbol{l}; \delta_1,\delta_1\tau,\ldots,\delta_1\tau^{i_1},\delta_1\delta_2\tau^{i_1},\delta_1\delta_2\tau^{i_1+1},\ldots,\delta_1\delta_2\tau^{i_1+i_2}).
\end{align*}
Similarly, we have
\begin{align*}
(\mathcal{C}_pF_2(z))_p&=\sum\limits_{\substack{i_1+i_2=a\\j_1+j_2=b\\i_1,i_2,j_1,j_2\geq0}}\alpha_2^{i_1}(\alpha_1+\alpha_2)^{i_2}\beta_2^{j_1}(\beta_1+\beta_2)^{j_2}\\
&\quad\times\sum_{\substack{|\boldsymbol{k}|_{i_1+1}=i_1+j_1+1\\|\boldsymbol{l}|_{i_2+1}= i_2+j_2+1}}M_{\mathcal{A}}(\boldsymbol{k},\boldsymbol{l}; \delta_2,\delta_2\tau,\ldots,\delta_2\tau^{i_1},\delta_1\delta_2\tau^{i_1},\delta_1\delta_2\tau^{i_1+1},\ldots,\delta_1\delta_2\tau^{i_1+i_2}).
\end{align*}
This completes the proof.
\qed

\begin{proofof}{Theorem \ref{Thm:main result}}
By comparing the coefficients of $\alpha_1^{a-m}\alpha_2^m\beta_1^{b-n}\beta_2^n$ on both sides of \eqref{Eq:wtsum}, we get the desired result.
\end{proofof}

\section{Applications of Theorem \ref{Thm:main result}}\label{Sec:Applications}
\subsection{Weighted sum formulas for FMTVs}\label{subsec:wtsum T-values}
 We begin with the proof of Theorem \ref{thm:wtsum-z-T}. Let $w,r$ be positive integers with $w\ge r$. For $0\leq i\leq r-1$ and $0\leq j\leq w-r$, set
\begin{align*}
R_{w,r}(i,j):=\sum_{\substack{|\boldsymbol{k}|_r=w\\ \mathcal{W}_{i+1}(\boldsymbol{k})=i+j+1}}T_\mathcal{A}(\boldsymbol{k}).
\end{align*}
Note that $R_{w,r}(i,j)$ is the sum of all finite multiple $T$-values of weight $w$ and depth $r$ for which the last $i+1$ components in  the index $\boldsymbol{k}$ have total weight $i+j+1$. For each fixed $0\leq i\leq r-1$, we have
\begin{align*}
\sum\limits_{j=0}^{w-r}R_{w,r}(i,j)=\sum_{\substack{|\boldsymbol{k}|_r=w}}T_\mathcal{A}(\boldsymbol{k}).
\end{align*}
Moreover, if $r$ is even, then the reversal relation \eqref{Eq:reversal relation} gives
\begin{align*}
R_{w,r}(i,j)=(-1)^wR_{w,r}(r-2-i,w-r-j)
\end{align*} 
for $0\le i\le r-2$ and $0\le j\le w-r$.

\begin{proofof}{Theorem \ref{thm:wtsum-z-T}}
Taking $$a=r-2,\qquad b=w-r,\qquad
(\delta_1,\delta_2,\tau)=(-1,-1,-1)$$
in Theorem \ref{Thm:main result}, we obtain 
\begin{align}
\sum_{i=0}^{r-2}\sum_{j=0}^{w-r}\left\{\binom{i}{m}\binom{j}{n}+\binom{i}{r-2-m}\binom{j}{w-r-n}\right\}R_{w,r}(i,j)=(-1)^{m+n+1}R_{w,r}(m,n)\label{eq:wtsum-cor-T}
\end{align}
for all integers $m,n$ satisfying $0\leq m\leq r-2$ and $0\leq n\leq w-r$. Multiplying \eqref{eq:wtsum-cor-T} by $z^{m+n}$, summing over $0\leq m\leq r-2$ and $0\leq n\leq w-r$, and using 
\begin{align*}
\sum_{m=0}^{r-2}\binom{i}{m}z^m=(1+z)^i,\quad\sum_{m=0}^{r-2}\binom{i}{r-2-m}z^m=z^{r-2-i}(1+z)^i,
\end{align*}
together with the analogous identities for $n$, we have
\begin{align}
\sum_{i=0}^{r-2}\sum_{j=0}^{w-r}
\left[(1+z)^{i+j}+z^{w-2-i-j}(1+z)^{i+j}+(-z)^{i+j}\right]R_{w,r}(i,j)=0.\label{eq:wtsum-z-T}
\end{align}
Putting $d=i+1$, the defining condition of $R_{w,r}(i,j)$ gives $\mathcal{W}_d(\boldsymbol{k})=i+j+1$. Therefore, we get the desired result.
\end{proofof}

By taking special values of $z$, one may obtain further weighted sum formulas involving the partial weights $\mathcal{W}_d(\boldsymbol{k})$. Taking $z=1$ in Theorem \ref{thm:wtsum-z-T} yields the following weighted sum formula.

\begin{cor}\
Let $w\geq r\geq2$ with $r$ even. Then
\begin{align*}
\sum_{|\boldsymbol{k}|_r=w}\sum_{d=1}^{r-1}\left(2^{\mathcal{W}_d(\boldsymbol{k})}-(-1)^{\mathcal{W}_d(\boldsymbol{k})}\right)T_\mathcal{A}(\boldsymbol{k})=0.
\end{align*}
\end{cor}
For example, when $(w,r)=(3,2)$, we obtain
$$T_\mathcal{A}(2,1)+T_\mathcal{A}(1,2)=0.$$
For $(w,r)=(4,2)$, we have
$$T_\mathcal{A}(3,1)+T_\mathcal{A}(2,2)+3T_\mathcal{A}(1,3)=0.$$
For $(w,r)=(5,4)$, we have
\begin{align*}
5T_\mathcal{A}(2,1,1,1)+7T_\mathcal{A}(1,2,1,1)+9T_\mathcal{A}(1,1,2,1)+9T_\mathcal{A}(1,1,1,2)=0.
\end{align*}
 Combining the above identity with the reversal relation \eqref{Eq:reversal relation}, one can see that all FMTVs of weight five and depth four are rational multiples of $T_{\mathcal A}(2,1,1,1)$.

Taking $z=-1$ in Theorem \ref{thm:wtsum-z-T} yields the following result.
\begin{cor}
Let $w\geq r\geq2$ with $r$ even. Then
\begin{align}
(r-1)\sum_{|\boldsymbol{k}|_r=w}T_\mathcal{A}(\boldsymbol{k})=-\bigl(1+(-1)^w\bigr)\sum_{\substack{|\boldsymbol{k}|_r=w\\ k_r=1}}T_\mathcal{A}(\boldsymbol{k}).\label{Eq:z=-1}
\end{align}
\end{cor}
\proof
 Setting $z=-1$ in \eqref{eq:wtsum-T}, the first term vanishes unless $\mathcal W_d(\boldsymbol{k})=1$, which occurs only when $d=1$ and $k_r=1$. Since $(-z)^{\mathcal W_d(\boldsymbol{k})-1}=1$, we obtain the desired result.
\qed

In particular, if $w$ is odd, then \eqref{Eq:z=-1} gives
\begin{align}
\sum_{|\boldsymbol{k}|_r=w}T_\mathcal{A}(\boldsymbol{k})=0\qquad(r\ \mathrm{even},\; w\ \mathrm{odd}).\label{eq:weight-odd sum}
\end{align}
One can see that \eqref{eq:weight-odd sum} also follows directly from the reversal relation \eqref{Eq:reversal relation}. If $w$ is even, \eqref{Eq:z=-1} gives a relation between the sum with fixed weight and depth and the sum over indices with $k_r=1$. For example, when $(w,r)=(6,2)$, we get
$$3T_\mathcal{A}(5,1)+T_\mathcal{A}(4,2)+T_\mathcal{A}(3,3)+T_\mathcal{A}(2,4)+T_\mathcal{A}(1,5)=0.$$
For $(w,r)=(5,4)$, we have
\begin{align*}
T_\mathcal{A}(2,1,1,1)+T_\mathcal{A}(1,2,1,1)+T_\mathcal{A}(1,1,2,1)+T_\mathcal{A}(1,1,1,2)=0.
\end{align*}

\subsection{Weighted sum formulas for depth-two FMMVs}\label{subsec:depth-two}
Let $a=0$ in Theorem \ref{Thm:main result}, we get the following weighted sum formula for depth-two FMMVs.
\begin{cor}\label{Cor:double-wtsum}
Let $b$ be a non-negative integer, and let $\delta_1,\delta_2\in\{\pm 1\}$. Then for any integer $n$ satisfying $0\leq n\leq b$, we have
  \begin{align}
  &\sum\limits_{j=0}^b\binom{j}{n}M_\mathcal{A}(b-j+1,j+1;\delta_1,\delta_1\delta_2)+\sum\limits_{j=0}^b\binom{j}{b-n}M_\mathcal{A}(b-j+1,j+1;\delta_2,\delta_1\delta_2)\notag\\
  &=(-1)^{n+1}M_\mathcal{A}(b-n+1,n+1;\delta_1,-\delta_2).\label{Eq:double-wtsum}
  \end{align}
\end{cor}
By considering all possible combinations of $\delta_1$ and $\delta_2$ in \eqref{Eq:double-wtsum}, we get the following weighted sum formulas, where the integer $n$ satisfies $0\leq n\leq b$. Let $(\delta_1,\delta_2)=(1,1)$, we obtain 
\begin{align}
&\sum\limits_{j=0}^b\left(\binom{j}{n}+\binom{j}{b-n}\right)E_\mathcal{A}(b-j+1,j+1)=(-1)^{n+1}S_\mathcal{A}(b-n+1,n+1).\label{Eq:E-double-wtsum}
\end{align}
Let $(\delta_1,\delta_2)=(1,-1)$, we get 
\begin{align}
&\sum\limits_{j=0}^b\binom{j}{n}S_\mathcal{A}(b-j+1,j+1)+\sum\limits_{j=0}^b\binom{j}{b-n}t_\mathcal{A}(b-j+1,j+1)\notag\\
  &=(-1)^{n+1}E_\mathcal{A}(b-n+1,n+1).\label{Eq:St-double-wtsum-1}
\end{align}
Let $(\delta_1,\delta_2)=(-1,1)$, we have 
\begin{align}
 &\sum\limits_{j=0}^b\binom{j}{n}t_\mathcal{A}(b-j+1,j+1)+\sum\limits_{j=0}^b\binom{j}{b-n}S_\mathcal{A}(b-j+1,j+1)\notag\\
  &=(-1)^{n+1}t_\mathcal{A}(b-n+1,n+1).\label{Eq:St-double-wtsum-2}
\end{align}
Let $(\delta_1,\delta_2)=(-1,-1)$, we obtain 
\begin{align}
 &\sum\limits_{j=0}^b\left(\binom{j}{n}+\binom{j}{b-n}\right)T_\mathcal{A}(b-j+1,j+1)=(-1)^{n+1}T_\mathcal{A}(b-n+1,n+1).\label{Eq:T-double-wtsum}
\end{align}
Note that \eqref{Eq:St-double-wtsum-1}  and \eqref{Eq:St-double-wtsum-2} are equivalent by the reversal relation and \eqref{Eq:T-double-wtsum} is precisely the depth-two specialization of \eqref{eq:wtsum-cor-T}.

Moreover, setting $n=0$ in \eqref{Eq:E-double-wtsum}, \eqref{Eq:St-double-wtsum-1},
\eqref{Eq:St-double-wtsum-2} and \eqref{Eq:T-double-wtsum} yields the following sum formulas, respectively:
\begin{align*}
&\sum\limits_{j=0}^{b}E_\mathcal{A}(b-j+1,j+1)=-S_\mathcal{A}(b+1,1)-E_\mathcal{A}(1,b+1),\\
 &\sum\limits_{j=0}^{b}S_\mathcal{A}(b-j+1,j+1)=-E_\mathcal{A}(b+1,1)-t_\mathcal{A}(1,b+1),\\
 &\sum\limits_{j=0}^{b}t_\mathcal{A}(b-j+1,j+1)=-t_\mathcal{A}(b+1,1)-S_\mathcal{A}(1,b+1),\\
 &\sum\limits_{j=0}^{b}T_\mathcal{A}(b-j+1,j+1)=-T_\mathcal{A}(b+1,1)-T_\mathcal{A}(1,b+1).
\end{align*}

\subsection{Relations for height-zero and height-one FMMVs}\label{subsec:height-zero and one}
In Theorem \ref{Thm:main result}, the indices have depth $a+2$ and weight $a+b+2$. Thus $b=0$ corresponds to the height-zero case, while $b=1$ corresponds to the height-one case.

Let $b=0$ in Theorem \ref{Thm:main result}, we get a weighted sum formula for height-zero FMMVs.
\begin{cor}\label{Cor:height-zero}
Let $a$ be a non-negative integer, and let $\delta_1,\delta_2,\tau\in\{\pm 1\}$. Then for any integer $m$ satisfying $0\leq m\leq a$, we have
  \begin{align}
 &\sum\limits_{i=0}^a\binom{i}{m}M_\mathcal{A}(\{1\}^{a+2};\delta_1,\delta_1\tau,\ldots,\delta_1\tau^{a-i},\delta_1\delta_2\tau^{a-i},\ldots,\delta_1\delta_2\tau^{a-1},\delta_1\delta_2\tau^{a})\notag\\
  &+\sum\limits_{i=0}^a\binom{i}{a-m}M_\mathcal{A}(\{1\}^{a+2};\delta_2,\delta_2\tau,\ldots,\delta_2\tau^{a-i},\delta_1\delta_2\tau^{a-i},\ldots,\delta_1\delta_2\tau^{a-1},\delta_1\delta_2\tau^{a})\notag\\
  &=(-1)^{m+1}M_\mathcal{A}(\{1\}^{a+2};\delta_1,\delta_1\tau,\ldots,\delta_1\tau^{a-m},-\delta_2\tau^{m},-\delta_2\tau^{m-1},\ldots,-\delta_2\tau,-\delta_2).\label{Eq:height-zero}
  \end{align}
\end{cor}

Several useful specializations of \eqref{Eq:height-zero} are given below. Considering $(\delta_1,\delta_2,\tau)=(1,1,1)$ in \eqref{Eq:height-zero}, we obtain 
\begin{align*}
M_\mathcal{A}(\{1\}^{a+2};\{1\}^{a-m+1},\{-1\}^{m+1})=(-1)^{m+1}\binom{a+2}{m+1}E_\mathcal{A}(\{1\}^{a+2}),\quad (0\leq m\leq a).
\end{align*}
It follows from the reversal relation \eqref{Eq:reversal relation} that $E_\mathcal{A}(\{1\}^{a+2})=(-1)^a t_\mathcal{A}(\{1\}^{a+2})$. Hence, we have
\begin{align*}
M_\mathcal{A}(\{1\}^{a+2};\{1\}^{a-m+1},\{-1\}^{m+1})=(-1)^{a+m+1}\binom{a+2}{m+1}t_\mathcal{A}(\{1\}^{a+2}),\quad (0\leq m\leq a).
\end{align*}
Thus, these identities express a family of height-zero FMMVs with a two-block sign pattern in terms of the FMtV $t_{\mathcal A}(\{1\}^{a+2})$ or the FMEV $E_{\mathcal A}(\{1\}^{a+2})$.
For example, taking $a=1$ and $m=0$ in the above identity gives $$M_\mathcal{A}(1,1,1;1,1,-1)=3t_\mathcal{A}(1,1,1).$$

Taking $(\delta_1,\delta_2,\tau)=(-1,-1,-1)$ in \eqref{Eq:height-zero}, we obtain 
  \begin{align*}
  &M_\mathcal{A}(\{1\}^{a+2};-1,1,\ldots,(-1)^{a-m+1},(-1)^m,(-1)^{m-1},\ldots,1)\\&=(-1)^{m+1}\binom{a+2}{m+1}T_\mathcal{A}(\{1\}^{a+2}),\quad (0\leq m\leq a).
  \end{align*}
If $a$ is even, then setting $m=0$ in the above identity gives $T_\mathcal{A}(\{1\}^{a+2})=0$. Equivalently, 
\begin{align}
T_\mathcal{A}(\{1\}^{r})=0,\quad(r~\text{even}).\label{eq:height-zero odd depth}
\end{align}
Note that \eqref{eq:height-zero odd depth} was also obtained by Zhao \cite[Proposition 7]{Zhao2024a}. For odd $r$, Zhao further conjectured that \begin{align*}
T_\mathcal{A}(\{1\}^{r})=\frac{2^{r-1}-1}{2^{r-2}}Z(r),\quad(r~\text{odd}).
\end{align*}
Thus, all height-zero FMTVs vanish in even depth, whereas in odd depth they are conjecturally proportional to $Z(r)$.

In the height-one case, since the weight exceeds the depth by one, every index occurring in this case is of the form $(\{1\}^{j},2,\{1\}^{a-j+1})$. Taking $b=1$, $n=0$ and $(\delta_1,\delta_2,\tau)=(-1,1,1)$ in Theorem \ref{Thm:main result}, we obtain the following sum formula for height-one FMtVs.
\begin{cor}\label{Cor:t-height-one}
Let $a$ be a non-negative integer. Then for any integer $m$ satisfying $0\leq m\leq a$, we have
  \begin{align}
 &\binom{a+1}{m+1}\sum\limits_{j=0}^{a+1}t_\mathcal{A}(\{1\}^{j},2,\{1\}^{a-j+1})+(-1)^m\sum\limits_{j=0}^{a-m}t_\mathcal{A}(\{1\}^{j},2,\{1\}^{a-j+1})\notag\\&=-\sum\limits_{i=0}^a\sum\limits_{j=0}^i\binom{i}{a-m}M_\mathcal{A}(\{1\}^{a-j+1},2,\{1\}^{j};\{1\}^{a-i+1},\{-1\}^{i+1}).\label{Eq:t-height-one}
  \end{align}
\end{cor}

Let $b=1$, $n=0$ and $(\delta_1,\delta_2,\tau)=(-1,-1,-1)$ in Theorem \ref{Thm:main result}, we get the following  weighted sum formula for height-one FMTVs. It is precisely the height-one specialization of \eqref{eq:wtsum-cor-T}. 
\begin{cor}\label{Cor:T-height-one}
Let $a$ be a non-negative even integer. Then for any integer $m$ satisfying $0\leq m\leq a$, we have
  \begin{align}
 &\sum\limits_{j=0}^{a+1}\left(\binom{a+2}{m+1}-\binom{a-j+1}{a-m+1}\right)T_\mathcal{A}(\{1\}^{j},2,\{1\}^{a-j+1})\notag\\&=(-1)^{m+1}\sum\limits_{j=0}^{a-m}T_\mathcal{A}(\{1\}^{j},2,\{1\}^{a-j+1}).\label{Eq:T-height-one}
  \end{align}
\end{cor}

For $b\ge2$, the situation becomes more complicated, since indices of different heights may occur simultaneously. Indeed, the indices in Theorem \ref{Thm:main result} satisfy
$$ \sum_{i=1}^{a+2}(k_i-1)=\wt(\boldsymbol{k})-\dep(\boldsymbol{k})=b. $$
For example, when \(b=2\), one encounters both height-one indices containing a single component \(3\) and height-two indices containing two components equal to \(2\). This suggests studying further specializations of Theorem \ref{Thm:main result} with fixed height, or more generally with a prescribed distribution of $b$ among the components, which may lead to relations for higher-height FMMVs.


\section*{Acknowledgments}
The first author is supported by the Natural Science Foundation of Shanghai (Grant No. 24ZR1469000).


\begin{thebibliography}{99}

\bibitem{Eie-Wei2008} M. Eie and C.-S. Wei, A short proof for the sum formula and its generalization, \textit{Arch. Math.} \textbf{91} (2008), 330-338.


\bibitem{GuoXie2009} L. Guo and B. Xie, Weighted sum formula for multiple zeta values, \textit{J. Number Theory} \textbf{129} (2009), 2747-2765. 

\bibitem{Hoffman1992} M. E. Hoffman, Multiple harmonic series, \textit{Pacific J. Math.} \textbf{152} (2) (1992), 275-290.

\bibitem{Hoffman2019} M. E. Hoffman, An odd variant of multiple zeta values, \textit{Comm. Number Theory Phys.} \textbf{13} (2019), 529-567.

\bibitem{Kaneko2019} M. Kaneko, An introduction to classical and finite multiple zeta values,
\textit{Publications math\'ematiques de Besan\c{c}on. Alg\`ebre et th\'eorie des nombres}, 2019/1, (2019), 103-129.

\bibitem{KanekoZagier-prep} M. Kaneko and D. Zagier, Finite multiple zeta values, in preparation.

\bibitem{Kaneko-Tsumura2020} M. Kaneko and H. Tsumura,  On a variant of multiple zeta values of level two, \textit{Tsukuba J. Math.} \textbf{44} (2) (2020), 213-234. 

\bibitem{Kamano2018} K. Kamano, Weighted sum formulas for finite multiple zeta values, \textit{J. Number Theory} \textbf{192} (2018), 168-180.

\bibitem{KanekoMurakamiYoshihara2023} M. Kaneko, T. Murakami and A. Yoshihara, On finite multiple zeta values of level two, \textit{Pure Appl. Math. Q.} \textbf{19} (1) (2023), 267-280.

\bibitem{LiWangZhang2026} Z. Li, Z. Wang and L. Zhang, A proof of a weighted sum conjecture for finite multiple zeta values of level two, arXiv: 2608.10675. 

\bibitem{OhnoZudilin} Y. Ohno and W. Zudilin, Zeta stars, \textit{Commun. Number Theory Phys.} \textbf{2} (2008), 325-347.

\bibitem{OngEieLiaw} Y. L. Ong, M. Eie and W.-C. Liaw, On generalizations of weighted sum formulas of multiple zeta values, \textit{Int. J. Number Theory} \textbf{9} (2013), 1185-1198.

\bibitem{SaitoWakabayashi} S. Saito and N. Wakabayashi, Sum formula for finite multiple zeta values, \textit{J. Math. Soc. Japan} \textbf{67} (3) (2015), 1069-1076.

\bibitem{XuZhao2022} C. Xu and J. Zhao, Variants of multiple zeta values with even and odd summation indices, \textit{Math. Z.} \textbf{300} (3) (2022), 3109-3142.

\bibitem{Zhao2016} J. Zhao, \textit{Multiple Zeta Functions, Multiple Polylogarithms and Their Special Values, Series on Number Theory and Its Applications}, Vol. 12, World Scientific Publishing Co. Pte. Ltd., Hackensack, NJ, 2016.

\bibitem{Zhao2024a} J. Zhao, Finite and symmetric Euler sums and finite and symmetric (alternating) multiple $T$-values, \textit{Axioms} \textbf{13} (2024), 210.

\bibitem{Zhao2024b} J. Zhao, Finite multiple mixed values, \textit{Foundations} \textbf{4} (2024), 451-467.


\bibitem{Zagier1994} D. Zagier, Values of zeta functions and their applications, in First European Congr. Math., vol. 2, {\em Progr. Math.} {\bf 120} (1994), 497-512.

\end{thebibliography}
\end{document}